\documentclass[a4paper,10pt,twoside]{article}
 
\usepackage{amsmath,amssymb,amsthm,dsfont}
\usepackage[english]{babel}
\usepackage[bookmarks]{hyperref} 
\usepackage{mathrsfs} 
\usepackage{color,soul} 

\usepackage{geometry}
\newtheorem{theorem}{Theorem}[section]
\newtheorem{corollary}{Corollary}

\newtheorem{lemma}[theorem]{Lemma}
\newtheorem{proposition}{Proposition}

\theoremstyle{definition}

\newtheorem{remark}{Remark}

\newcommand{\N}{\mathbb N}
\newcommand{\Z}{\mathbb Z}

\newcommand{\R}{\mathbb R}
\newcommand{\C}{\mathbb C}

\newcommand{\Sh}{\mathscr{S}}

\newcommand{\ep}{\epsilon}
\newcommand{\Gact} {\gamma_{\text{c}}}
\newcommand{\Gace} {\gamma_{\emph{c}}}

\newcommand{\scal}[1]{\left\langle #1 \right\rangle} 

\newcommand{\defendproof}{\hfill $\Box$} 

\usepackage{fancyhdr} 
\begin{document}
\title{\sc On the Cauchy problem for the nonlinear semi-relativistic equation in Sobolev spaces}
\author{\sc{Van Duong Dinh}} 
\date{ }
\maketitle

\begin{abstract}
We proved the local well-posedness for the power-type nonlinear semi-relativistic or half-wave equation (NLHW) in Sobolev spaces. Our proofs mainly bases on the contraction mapping argument using Strichartz estimate. We also apply the technique of Christ-Colliander-Tao in \cite{ChristCollianderTao} to prove the ill-posedness for (NLHW) in some cases of the super-critical range. 
\end{abstract}


\section{Introduction}
We consider the Cauchy problem for the semi-relativistic or half-wave equation posed on $\R^d, d\geq 1$, namely
\begin{align}
\left\{
\begin{array}{rcl}
i\partial_t u(t,x) + \Lambda u(t,x)&=&-\mu |u|^{\nu-1} u(t,x), \quad (t, x) \in \R \times \R^d, \\
u(0,x) &=& u_0(x), \quad x\in \R^d.
\end{array}
\right.
\tag{NLHW}
\end{align}
where $\nu>1, \mu\in \{\pm 1\}$ and $\Lambda=\sqrt{-\Delta}$ is the Fourier multiplier of symbol $|\xi|$. The number $\mu=1$ (resp. $\mu=-1$) corresponds to the defocusing case (resp. focusing case). The Cauchy problem such as (NLHW) arises in various physical contexts, such as water waves (see e.g. \cite{IonescuPusateri}), and the gravitational collapse (see e.g. \cite{ElgartSchlein}, \cite{FrohlichLenzmann}).\newline
\indent It is worth noticing that if we set for $\lambda>0$, 
\[
u_\lambda(t,x)= \lambda^{-\frac{1}{\nu-1}} u( \lambda^{-1} t, \lambda^{-1} x),
\]
then the (NLHW) is invariant under this scaling, i.e. for $T \in (0, +\infty]$, $u$ solves the (NLHW) on $(-T, T)$ is equivalent to $u_\lambda$ solves the (NLHW) on $(-\lambda T, \lambda T)$ with initial data $u_\lambda(0)= \lambda^{-\frac{1}{\nu-1}} u_0(\lambda^{-1}x)$. We also have
\[
\|u_\lambda(0)\|_{\dot{H}^\gamma} = \lambda^{\frac{d}{2}-\frac{1}{\nu-1}- \gamma} \|u_0\|_{\dot{H}^\gamma}. \nonumber
\]
From this, we define the critical regularity exponent for the (NLHW) by
\begin{align}
\Gact =\frac{d}{2} - \frac{1}{\nu-1}. \label{critical half wave}
\end{align}
One says that $H^\gamma$ is sub-critical (critical, super-critical) if $\gamma>\Gact$ ($\gamma=\Gact$, $\gamma<\Gact$) respectively. Another important property of the (NLHW) is that the following mass and energy are formally conserved under the flow of the equation,
\[
M(u(t))= \int |u(t,x)|^2 dx, \quad E(u(t))= \int \frac{1}{2}|\Lambda^{1/2} u(t,x)|^2 + \frac{\mu}{\nu+1}|u(t,x)|^{\nu+1} dx.
\] 
\indent The nonlinear half-wave equation (NLHW) has attracted a lot of works in a past decay (see e.g. \cite{FrohlichLenzmann}, \cite{KriegerLenzmannRaphael}, \cite{FujiwaraOzawa}, \cite{ChoffrutPocovnicu}, \cite{FujiwaraGeorgievOzawa} and references therein). The main purpose of this note is to give the local well-posedness and ill-posedness results for the (NLHW) in Sobolev spaces. It can be seen as an complement to the recent results of Hong-Sire \cite{HongSire} where the authors considered the same problem for the nonlinear fractional Schr\"odinger equation on $\R^d, d\geq 1$, namely
\[
i\partial_t u + (-\Delta)^\sigma u + \mu |u|^{\nu-1}u =0, \quad \sigma \in (0,1)\backslash \{1/2\}. \tag{NLFS}
\]
The proofs of the local well-posedness results in \cite{HongSire} are based on Strichartz estimates, which are similar to those for the classical nonlinear Schr\"odinger equation, but with a loss of derivative. In the sub-critical case, the derivative loss is compensated for by using Sobolev embeddings. In the critical case, the Sobolev embedding does not help. To remove the derivative loss, the authors used Strichartz norms localized in dyadic pieces, and then summed up in a $\ell^2$-fashion. For the ill-posedness result in \cite{HongSire}, the pseudo-Galilean transformation is expoited. This transformation does not provide an invariance of the equation, but nevertheless, in a certain setting, the error of the psedo-Galilean transformation can be controlled. This allows the authors to prove the ill-posedness for the (NLFS) in $H^\gamma$ with a certain range of the negative exponent $\gamma$. \newline
\indent The proofs of the well-posedness results in this note are also based on Strichartz estimates and the standard contraction argument. We thus only focus on the case $d\geq 2$ where Strichartz estimates are available and just recall known results in the case $d=1$. More precisely, in the sub-critical case, we prove the well-posedness in $H^\gamma$ with $\gamma$ satisfying
\begin{align}
\left\{
\begin{array}{ll}
\gamma > 1-1/\max(\nu-1,4) & \text{when } d=2, \\
\gamma > d/2-1/\max(\nu-1,2) &\text{when } d\geq 3,
\end{array}
\right. \label{subcritical assuption}
\end{align}
and if $\nu$ is not an odd integer, 
\begin{align}
\lceil \gamma\rceil \leq \nu, \label{assumption smoothness subcritical}
\end{align}
where $\lceil \gamma \rceil$ is the smallest positive integer greater than or equal to $\gamma$. This remains a gap between $\Gact$ and $1-1/\max(\nu-1,4)$ when $d=2$ and $ d/2-1/\max(\nu-1,2)$ when $d\geq3$. The proof of this result also makes use of the Sobolev embedding, which is similar to that in \cite{HongSire}. The point here is that we present a simple proof and a better result comparing to those of \cite{HongSire}. In the critical case, we can successfully apply the argument of \cite{HongSire} (see also \cite{Dinh}) to prove the local well-posedness with small data scattering in $H^{\Gact}$ provided $\nu>5$ for $d=2$ and $\nu>3$ for $d\geq 3$ and if $\nu$ is not an odd integer, 
\begin{align}
\lceil \Gact\rceil \leq \nu. \label{assumption smoothness critical}
\end{align}
The cases $\nu \in (1,5]$ when $d=2$ and $\nu\in (1,3]$ when $d\geq 3$ still remain open. It requires another technique rather than just Strichartz estimates. Note that conditions $(\ref{assumption smoothness subcritical})$ and $(\ref{assumption smoothness critical})$ allow the nonlinearity to have enough of regularity to apply the fractional derivative estimates (see Subsection $\ref{subsection fractional derivative}$). Finally, using the technique of Christ-Colliander-Tao given in \cite{ChristCollianderTao}, we are able to prove the ill-posedness for the (NLHW) in a certain range of the super-critical case. More precisely, we prove the ill-posedness in $H^\gamma$ with 
\begin{align}
\left\{
\begin{array}{ll}
\gamma \in (-\infty,-d/2] \cup [0,\Gact) & \text{when } \Gact>0, \\
\gamma \in (-\infty,-d/2] \cap (-\infty, \Gact) & \text{otherwise}.
\end{array}
\right. \label{condition ill posedness}
\end{align}
We expect that the ill-posedness still holds in the range $\gamma \in (-d/2,\max\{0,\Gact\})$ as for the classical nonlinear Schr\"odinger equation (see \cite{ChristCollianderTao}). But it is not clear to us how to prove it at the moment. As mentioned above, the authors in \cite{HongSire} used the technique of \cite{ChristCollianderTao} with the pseudo-Galilean transformation to prove the ill-posedness for the (NLFS) with negative exponent. Unfortunately, it seem to be difficult to control the error of the pseudo-Galilean transformation in high Sobolev norms. As one can check from the ill-posedness result of \cite{HongSire} (Theorem 1.5 there), this result holds only in the one dimensional case and does not hold in the two and three dimensional cases as claimed there. Moreover, the dependence of $\upsilon$ in the estimate (5.7) of \cite{HongSire} seems to be eliminated which may effect the proof of the ill-posedness given in \cite{HongSire}. We thus do not persuade the technique of \cite{HongSire} to prove the ill-posedness result for the (NLHW) in the range $\gamma \in (-d/2,\max\{0,\Gact\})$. We end this paragraph by noting that the techniques used in this note can be applied without any difficulty for the (NLFS). It somehow provides better results for those of \cite{HongSire}.  \newline
\indent Let us now recall known results about the local well-posedness and the ill-posedness for the (NLHW) in 1D. It is well-known that the (NLHW) is locally well-posed in $H^\gamma(\R)$, with $\gamma>1/2$ satisfying $(\ref{assumption smoothness subcritical})$ if $\nu$ is not an odd integer, by using the energy method and the contraction mapping argument. When $\nu=3$, i.e. cubic nonlinearity, the (NLHW) is locally well-posed in $H^\gamma(\R)$ with $\gamma\geq 1/2$ (see e.g. \cite{KriegerLenzmannRaphael}, \cite{Pocovnicu}). This result is optimal in the sense that the equation is ill-posed in $H^\gamma(\R)$ provided $\gamma<1/2$ (see e.g. \cite{ChoffrutPocovnicu}). The proof of this ill-posedness result is mainly based on the relation with the cubic Szeg\"o equation, which can not be easily extended to general nonlinearity. To our knowledge, the local well-posedness for the generalized (NLHW) in $H^\gamma(\R)$ with $\gamma \leq 1/2$ seems to be an open question. \newline 
\indent Before stating our results, let us introduce some notations (see \cite[Appendix]{GinibreVelo85}, \cite{Triebel} or \cite{BerghLosfstom}). Let $\varphi_0 \in C^\infty_0(\R^d)$ be such that $\varphi_0(\xi)=1$ for $|\xi|\leq 1$ and $\text{supp}(\varphi_0) \subset \{\xi \in \R^d, |\xi|\leq 2\}$. Set $\varphi(\xi):= \varphi_0(\xi)-\varphi_0(2\xi)$. We see that $\varphi\in C^\infty_0(\R^d)$ and $\text{supp}(\varphi)\subset \{\xi \in \R^d, 1/2 \leq |\xi|\leq 2 \}$. We denote the Littlewood-Paley projections by $P_0:=\chi_0(D), P_N:=\chi(N^{-1}D)$ with $N=2^k, k\in \Z$ where $\chi_0(D), \chi(N^{-1}D)$ are the Fourier multipliers by $\chi_0(\xi)$ and $\chi(N^{-1}\xi)$ respectively. Given $\gamma \in \R$ and $1 \leq q \leq \infty$, the Sobolev and Besov spaces are defined by
\begin{align*}
H^\gamma_q &:= \Big\{ u \in \Sh' \ | \  \|u\|_{H^\gamma_q}:=\|\scal{\Lambda}^\gamma u\|_{L^q} <\infty \Big\}, \quad \scal{\Lambda}:=\sqrt{1+\Lambda^2}, \\
B^\gamma_q&:= \Big\{ u \in \Sh' \ | \ \|u\|_{B^\gamma_q}:= \|P_0 u\|_{L^q} + \Big( \sum_{N\in 2^{\Z}} N^{2\gamma} \|P_N u\|^2_{L^q} \Big)^{1/2} <\infty \Big\},
\end{align*}
where $\Sh'$ is the space of tempered distributions. Now let $\Sh_0$ be a subspace of the Schwartz space $\Sh$ consisting of functions $\phi$ satisfying $D^\alpha \hat{\phi}(0)=0$ for all $\alpha \in \N^d$ where $\hat{\cdot}$ is the Fourier transform on $\Sh$ and $\Sh'_0$ is its topology dual space. One can see $\Sh'_0$ as $\Sh'/\mathscr{P}$ where $\mathscr{P}$ is the set of all polynomials on $\R^d$. The homogeneous Sobolev and Besov spaces are defined by
\begin{align*}
\dot{H}^\gamma_q &:= \Big\{ u \in \Sh'_0 \ | \  \|u\|_{\dot{H}^\gamma_q}:=\|\Lambda^\gamma u\|_{L^q} <\infty \Big\}, \\
\dot{B}^\gamma_q &:= \Big\{ u \in \Sh'_0 \ | \  \|u\|_{\dot{B}^\gamma_q}:=\Big( \sum_{N\in 2^{\Z}} N^{2\gamma} \|P_N u\|^2_{L^q} \Big)^{1/2} <\infty \Big\}.
\end{align*}
It is easy to see that the norms $\|u\|_{B^\gamma_q}$ and $\|u\|_{\dot{B}^\gamma_q}$ do not depend on the choice of $\varphi_0$, and $\Sh_0$ is dense in $\dot{H}^\gamma_q, \dot{B}^\gamma_q$. Under these settings, $H^\gamma_q, B^\gamma_q, \dot{H}^\gamma_q$ and $\dot{B}^\gamma_q$ are Banach spaces with the norms $\|u\|_{H^\gamma_q}, \|u\|_{B^\gamma_q}, \|u\|_{\dot{H}^\gamma_q}$ and $\|u\|_{\dot{B}^\gamma_q}$ respectively (see e.g. \cite{Triebel}). In this note, we shall use $H^\gamma:= H^\gamma_2$, $\dot{H}^\gamma:= \dot{H}^\gamma_2$. We note (see \cite{BerghLosfstom}, \cite{GinibreVelo85}) that if $2\leq q <\infty$, then $\dot{B}^\gamma_q \subset \dot{H}^\gamma_q$. The reverse inclusion holds for $1<r\leq 2$. In particular, $\dot{B}^\gamma_2= \dot{H}^\gamma$ and $\dot{B}^0_2=\dot{H}^0_2=L^2$. Moreover, if $\gamma>0$, then $H^\gamma_q = L^q \cap \dot{H}^\gamma_q$ and $B^\gamma_q = L^q \cap \dot{B}^\gamma_q$. \newline
\indent In the sequel, a pair $(p,q)$ is said to be admissible if
\[
(p,q)\in [2,\infty]^2, \quad (p,q,d)\ne (2,\infty,3), \quad \frac{2}{p}+\frac{d-1}{q}\leq \frac{d-1}{2}.
\]
We also denote for $(p,q) \in [1,\infty]^2$,
\begin{align}
\gamma_{p,q}=\frac{d}{2}-\frac{d}{q}-\frac{1}{p}. \label{define gamma pq}
\end{align}
Our first result concerns with the local well-posedness for the (NLHW) in the sub-critical case.
\begin{theorem} \label{theorem subcritical}
Let $\gamma \geq 0$ and $\nu>1$ be such that $(\ref{subcritical assuption})$ holds, and also, if $\nu$ is not an odd integer, $(\ref{assumption smoothness subcritical})$. Then for all $u_0 \in H^\gamma$, there exist $T^* \in (0,\infty]$ and a unique solution to the \emph{(NLHW)} satisfying
\[
u \in C([0,T^*), H^\gamma) \cap L^p_{\emph{loc}}([0,T^*), L^\infty),
\]
for some $p> \max(\nu-1,4)$ when $d=2$ and some $p>\max(\nu-1,2)$ when $d\geq 3$. Moreover, the following properties hold:
\begin{itemize}
\item[i.] If $T^*<\infty$, then $\|u(t)\|_{H^\gamma} \rightarrow \infty$ as $t\rightarrow T^*$.
\item[ii.] $u$ depends continuously on $u_0$ in the following sense. There exists $0< T< T^*$ such that if $u_{0,n} \rightarrow u_0$ in $H^\gamma$ and if $u_n$ denotes the solution of the \emph{(NLHW)} with initial data $u_{0,n}$, then $0<T< T^*(u_{0,n})$ for all $n$ sufficiently large and $u_n$ is bounded in $L^a([0,T],H^{\gamma-\gamma_{a,b}}_b)$ for any admissible pair $(a,b)$ with $b<\infty$. Moreover, $u_n \rightarrow u$ in $L^a([0,T],H^{-\gamma_{a,b}}_b)$ as $n \rightarrow \infty$. In particular, $u_n \rightarrow u$ in $C([0,T],H^{\gamma-\ep})$ for all $\ep>0$.
\item[iii.] Let $\beta>\gamma$ be such that if $\nu$ is not an odd integer, $\lceil\beta\rceil \leq \nu$. If $u_0 \in H^\beta$, then $u \in C([0,T^*),H^\beta)$.
\end{itemize}
\end{theorem}
The continuous dependence can be improved to hold in $C([0,T],H^\gamma)$ if we assume that $\nu>1$ is an odd integer or $\lceil\gamma\rceil \leq \nu-1$ otherwise (see Remark $\ref{rem continuous dependence}$).  We also have the following local well-posedness with small data scattering in the critical case.
\begin{theorem} \label{theorem critical}
Let 
\begin{align}
\left\{
\begin{array}{ll}
\nu>5 & \text{when } d=2, \\
\nu > 3 &\text{when } d\geq 3,
\end{array}
\right. \label{critical assuption}
\end{align}
and also, if $\nu$ is not an odd integer, $(\ref{assumption smoothness critical})$. Then for all $u_0 \in H^{\Gace}$, there exist $T^* \in (0,\infty]$ and a unique solution to the \emph{(NLHW)} satisfying
\[
u \in C([0,T^*),H^{\Gace}) \cap L^p_{\emph{loc}} ([0,T^*),B^{\Gace-\gamma_{p,q}}_q),
\]
where $p=4, q= \infty$ when $d=2$; $2<p<\nu-1, q=p^\star=2p/(p-2)$ when $d=3$; $p=2, q=2^\star=2(d-1)/(d-3)$ when $d\geq 4$.  
Moreover, if $\|u_0\|_{\dot{H}^{\Gace}}<\varepsilon$ for some $\varepsilon>0$ small enough, then $T^*=\infty$ and the solution is scattering in $H^{\Gace}$, i.e. there exists $u^+_0 \in H^{\Gace}$ such that 
\[
\lim_{t\rightarrow +\infty} \|u(t)-e^{it\Lambda} u^+_0 \|_{H^{\Gace}}=0.
\]
\end{theorem}
Our final result is the following ill-posedness for the (NLHW).
\begin{theorem} \label{theorem ill-posedness}
Let $\nu>1$ be such that if $\nu$ is not an odd integer, $\nu\geq k+1$ for some integer $k>d/2$. Then the \emph{(NLHW)} is ill-posed in $H^\gamma$ with $\gamma$ satisfying $(\ref{condition ill posedness})$. More precisely, if  $\gamma \in (-\infty,-d/2] \cup (0,\Gact)$ when $\Gact>0$ or $\gamma \in (-\infty,-d/2] \cap (-\infty, \Gact)$ otherwise, then for any $t>0$ the solution map $\Sh \ni u(0)\mapsto u(t)$ of the \emph{(NLHW)} fails to be continuous at 0 in the $H^\gamma$ topology. Moreover, if $\Gace>0$, the solution map fails to be uniformly continuous on $L^2$.
\end{theorem}
\indent This note is organized as follows. In Section 2, after recalling Strichartz estimates for the linear half-wave equation and nonlinear fractional derivative estimates, we prove the local well-posedness given in Theorem $\ref{theorem subcritical}$ and Theorem $\ref{theorem critical}$. The proof of the ill-posedness will be given in Section 3.  
\section{Local well-posedness}
In this section, we will give the proofs of Theorem $\ref{theorem subcritical}$ and Theorem $\ref{theorem critical}$. Our proofs are based on the standard contraction mapping argument using Strichartz estimates and nonlinear fractional derivatives. 
\subsection{Linear estimates}
In this subsection, we recall Strichartz estimates for the linear half-wave equation. 
\begin{theorem}[\cite{BCDfourier}, \cite{KeelTaoTTstar}] \label{theorem strichartz}
Let $d\geq 2, \gamma \in \R$ and $u$ be a (weak) solution to the linear half-wave equation, namely
\[
u(t)=e^{it\Lambda}u_0+\int_0^t e^{i(t-s)\Lambda} F(s)ds,
\]
for some data $u_0, F$. Then for all $(p,q)$ and $(a,b)$ admissible pairs,
\begin{align}
\|u\|_{L^p(\R, \dot{B}^\gamma_q)} \lesssim \|u_0\|_{\dot{H}^{\gamma+\gamma_{p,q}}} + \|F\|_{L^{a'}(\R, \dot{B}^{\gamma+\gamma_{p,q}-\gamma_{a',b'}-1}_{b'})}, \label{strichartz estimate Besov}
\end{align}
where $\gamma_{p,q}$ and $\gamma_{a',b'}$ are as in $(\ref{define gamma pq})$. In particular,
\begin{align}
\|u\|_{L^p(\R, \dot{B}^{\gamma-\gamma_{p,q}}_q)} \lesssim \|u_0\|_{\dot{H}^\gamma} + \|F\|_{L^1(\R, \dot{H}^\gamma)}. \label{homogeneous strichartz}
\end{align}
Here $(a,a')$ and $(b,b')$ are conjugate pairs.
\end{theorem}
The proof of this result is based on the scaling technique. We refer the reader to \cite[Section 8.3]{BCDfourier} for more details. 
\begin{corollary} \label{coro strichartz}
Let $d\geq 2$ and $\gamma \in \R$. If $u$ is a (weak) solution to the linear half-wave equation for some data $u_0, F$, then for all $(p,q)$ admissible satisfying $q<\infty$,
\begin{align}
\|u\|_{L^p(\R,H^{\gamma-\gamma_{p,q}}_q)} \lesssim \|u_0\|_{H^\gamma}+\|F\|_{L^1(\R,H^\gamma)}. \label{strichartz sobolev}
\end{align}
\end{corollary} 
\begin{proof}
We firstly remark that $(\ref{homogeneous strichartz})$ together with the Littlewood-Paley theorem yield for any $(p,q)$ admissible satisfying $q<\infty$,
\begin{align}
\|u\|_{L^p(\R,\dot{H}^{\gamma-\gamma_{p,q}}_q)} \lesssim \|u_0\|_{\dot{H}^\gamma}+\|F\|_{L^1(\R,\dot{H}^\gamma)}. \label{strichartz 1}
\end{align}
We next write $\|u\|_{L^p(\R, H^{\gamma-\gamma_{p,q}}_q)}= \|\scal{\Lambda}^{\gamma-\gamma_{p,q}} u\|_{L^p(\R, L^q)}$ and apply $(\ref{strichartz 1})$ with $\gamma=\gamma_{p,q}$ to get
\[
\|u\|_{L^p(\R, H^{\gamma-\gamma_{p,q}}_q)} \lesssim \|\scal{\Lambda}^{\gamma-\gamma_{p,q}} u_0 \|_{\dot{H}^{\gamma_{p,q}}} + \|\scal{\Lambda}^{\gamma-\gamma_{p,q}}F\|_{L^1(\R, \dot{H}^{\gamma_{p,q}})}. 
\]
The estimate $(\ref{strichartz sobolev})$ then follows by using the fact that $\gamma_{p,q}>0$ for all $(p,q)$ is admissible satisfying $q<\infty$. 
\end{proof}
\subsection{Nonlinear estimates} \label{subsection fractional derivative}
In this subsection, we recall some nonlinear fractional derivative estimates related to our purpose. Let us start with the following fractional Leibniz rule (or Kato-Ponce inequality).
\begin{proposition} \label{prop fractional leibniz}
Let $\gamma \geq 0, 1<r<\infty$ and $1<p_1, p_2, q_1, q_2 \leq \infty$ satisfying
\[
\frac{1}{r}=\frac{1}{p_1}+\frac{1}{q_1}=\frac{1}{p_2}+\frac{1}{q_2}.
\]
Then there exists $C=C(d,\gamma, r, p_1, q_1, p_2, q_2)>0$  such that for all $u, v \in \Sh$,
\begin{align}
\|\Lambda^\gamma(uv)\|_{L^r} &\leq C \Big(\|\Lambda^\gamma u\|_{L^{p_1}} \|v\|_{L^{q_1}} + \|u\|_{L^{p_2}}\|\Lambda^\gamma v\|_{L^{q_2}} \Big), \label{leibniz homogeneous sobolev} \\
\|\scal{\Lambda}^\gamma(uv)\|_{L^r} &\leq C \Big(\|\scal{\Lambda}^\gamma u\|_{L^{p_1}} \|v\|_{L^{q_1}} + \|u\|_{L^{p_2}}\|\scal{\Lambda}^\gamma v\|_{L^{q_2}} \Big). \label{leibniz inhomogeneous sobolev} 
\end{align}
\end{proposition}
We refer to \cite{GrafakosOh} for the proof of above inequalities and more general results. We also have the following fractional chain rule. 
\begin{proposition}\label{prop fractional chain}
Let $F \in C^1(\C, \C)$ and $G \in C(\C, \R^+)$ such that $F(0)=0$ and 
\[
|F'( \theta z+ (1-\theta) \zeta)| \leq \mu(\theta) (G(z)+G(\zeta)), \quad z,\zeta \in \C, \quad 0 \leq \theta \leq 1,
\]
where $\mu \in L^1((0,1))$. Then for $\gamma \in (0,1)$ and $1 <r, p <\infty$, $1 <q \leq \infty$ satisfying 
\[
\frac{1}{r}=\frac{1}{p}+\frac{1}{q},
\] 
there exists $C=C(d,\mu, \gamma, r, p,q)>0$ such that for all $u \in \Sh$,
\begin{align}
\|\Lambda^\gamma F(u) \|_{L^r} &\leq C \|F'(u) \|_{L^q} \|\Lambda^\gamma u \|_{L^p},\label{chain homogeneous sobolev} \\
\|\scal{\Lambda}^\gamma F(u) \|_{L^r} &\leq C \|F'(u) \|_{L^q} \|\scal{\Lambda}^\gamma u \|_{L^p}. \label{chain inhomogeneous sobolev} 
\end{align}
\end{proposition}
We refer the reader to \cite{ChristWeinstein} (see also \cite{Staffilani}) for the proof of $(\ref{chain homogeneous sobolev})$ and \cite[Proposition 5.1]{Taylor} for $(\ref{chain inhomogeneous sobolev})$. Combining the fractional Leibniz rule and the fractional chain rule, one has the following result (see \cite[Appendix]{Kato95}).
\begin{lemma} \label{lem nonlinear estimates}
Let $F \in C^k(\C, \C), k \in \N \backslash \{0\}$. Assume that there is $\nu \geq k$  such that 
\[
|D^iF (z)| \leq C |z|^{\nu-i}, \quad z \in \C, \quad i=1,2,...., k.
\]
Then for $\gamma \in [0,k]$ and $1 <r, p <\infty$, $1 <q \leq \infty$ satisfying $\frac{1}{r}=\frac{1}{p}+\frac{\nu-1}{q}$, there exists $C=C(d, \nu, \gamma, r,p,q)>0$ such that for all $u \in \Sh$,
\begin{align}
\|\Lambda^\gamma F(u)\|_{L^r} &\leq C \|u\|^{\nu-1}_{L^q} \|\Lambda^\gamma u \|_{L^p}, \label{nonlinear homogeneous sobolev} \\
\|\scal{\Lambda}^\gamma F(u) \|_{L^r} & \leq C\|u\|^{\nu-1}_{L^q} \|\scal{\Lambda}^\gamma u\|_{L^p}. \label{nonlinear inhomogeneous sobolev}
\end{align}
Moreover, if $F$ is a polynomial in $u$ and $\overline{u}$, then $(\ref{nonlinear homogeneous sobolev})$ and $(\ref{nonlinear inhomogeneous sobolev})$ hold true for any $\gamma \geq 0$.
\end{lemma}
\begin{corollary} \label{coro fractional derivatives}
Let $F(z)=|z|^{\nu-1}z$ with $\nu>1$, $\gamma \geq 0$ and $1 <r, p <\infty$, $1 <q \leq \infty$ satisfying $\frac{1}{r}=\frac{1}{p}+\frac{\nu-1}{q}$. 
\begin{itemize}
\item[i.] If $\nu$ is an odd integer or, otherwise, if $\lceil \gamma \rceil \leq \nu$, then there exists $C=C(d,\nu, \gamma, r, p, q)>0$ such that for all $u \in \Sh$,
\[
\|F(u)\|_{\dot{H}^\gamma_r}  \leq C \|u\|^{\nu-1}_{L^q} \|u\|_{\dot{H}^\gamma_p}.
\]
A similar estimate holds with $\dot{H}^\gamma_r, \dot{H}^\gamma_p$-norms are replaced by $H^\gamma_r, H^\gamma_p$-norms respectively.
\item[ii.] If $\nu$ is an odd integer or, otherwise, if $\lceil \gamma \rceil \leq \nu$, then there exists $C= C(d,\nu,\gamma, r, p, q)>0$ such that for all $u, v \in \Sh$,
\begin{multline}
\|F(u)-F(v)\|_{\dot{H}^\gamma_r}  \leq C \Big( (\|u\|^{\nu-1}_{L^q} + \|v\|^{\nu-1}_{L^q}) \|u-v\|_{\dot{H}^\gamma_p} \Big. \\
\Big.+ (\|u\|^{\nu-2}_{L^q} +\|v\|^{\nu-2}_{L^q})(\|u\|_{\dot{H}^\gamma_p} + \|v\|_{\dot{H}^\gamma_p}) \|u-v\|_{L^q} \Big). \nonumber
\end{multline}
A similar estimate holds with $\dot{H}^\gamma_r, \dot{H}^\gamma_p$-norms are replaced by $H^\gamma_r, H^\gamma_p$-norms respectively.
\end{itemize}
\end{corollary}
A next result will give a good control on the nonlinear term which allows us to use the contraction mapping argument.
\begin{lemma} \label{lem nonlinear control}
Let $\nu$ be as in \emph{Theorem } $\ref{theorem critical}$ and $\Gace$ as in $(\ref{critical half wave})$. Then
\[
\|u\|^{\nu-1}_{L^{\nu-1}(\R, L^\infty)} \lesssim
\left\{ 
\begin{array}{ll}
\|u\|^4_{L^4(\R,\dot{B}^{\Gace-\gamma_{4,\infty}}_{\infty}} \|u\|^{\nu-5}_{L^\infty(\R,\dot{B}^{\Gace}_{2})} &\text{ when } d=2, \\
\|u\|^p_{L^p(\R,\dot{B}^{\Gace-\gamma_{p,p^\star}}_{p^\star})}\|u\|^{\nu-1-p}_{L^\infty(\R,\dot{B}^{\Gace}_{2})}  &\text{ when } d=3, \\
\|u\|^2_{L^2(\R,\dot{B}^{\Gace-\gamma_{2,2^\star}}_{2^\star})}\|u\|^{\nu-3}_{L^\infty(\R,\dot{B}^{\Gace}_{2})}  &\text{ when } d\geq 4,
\end{array}
\right.
\] 
where $2<p<\nu-1, p^\star= 2p/(p-2)$ and $2^\star= 2(d-1)/(d-3)$.
\end{lemma} 
The above lemma follows the same spirit as in \cite[Lemma 3.5]{HongSire} (see also \cite{Dinh}) using the argument of \cite[Lemma 3.1]{CoKeStaTaTao}. 
\begin{proof}
We only give a sketch of the proof in the case $d\geq 4$, the cases $d=2, 3$ are treated similarly. By interpolation, we can assume that $\nu-1= m/n >2, m,n \in \N$ with $\gcd(m,n)=1$. We proceed as in \cite{HongSire} and set
\[
c_N(t)= N^{\Gact-\gamma_{2,2^\star}}\|P_N u(t)\|_{L^{2^\star}(\R^d)}, \quad c'_N(t) = N^{\Gact}\|P_N u(t)\|_{L^2(\R^d)}.
\] 
By Bernstein's inequality, we have
\begin{align}
\|P_Nu(t)\|_{L^\infty(\R^d)}  &\lesssim N^{\frac{d}{2^\star}-\Gact+\gamma_{2,2^\star}} c_N(t)=N^{\frac{n}{m}-\frac{1}{2}} c_N(t), \label{first estimate} \\
\|P_Nu(t)\|_{L^\infty(\R^d)} &\lesssim  N^{\frac{d}{2}-\Gact} c'_N(t)=N^{\frac{n}{m}} c'_N(t). \nonumber
\end{align}
This implies that for $\theta \in (0,1)$ which will be chosen later,
\begin{align}
\|P_Nu(t)\|_{L^\infty(\R^d)} \lesssim  N^{\frac{n}{m}-\frac{\theta}{2}} (c_N(t))^\theta (c'_N(t))^{1-\theta}. \label{second estimate}
\end{align}
We next use
\[
A(t) := \Big( \sum_{N \in 2^{\Z}} \|P_Nu(t)\|_{L^\infty(\R^d)}\Big)^m \lesssim \sum_{N_1 \geq \cdots \geq N_m} \prod_{j=1}^{m} \|P_{N_j}u(t)\|_{L^\infty(\R^d)}.
\]
Estimating the $n$ highest frequencies by $(\ref{first estimate})$ and the rest by $(\ref{second estimate})$, we get
\[
A(t) \lesssim \sum_{N_1 \geq \cdots \geq N_m} \Big( \prod_{j=1}^{n}N_j^{\frac{n}{m} -\frac{1}{2}} c_{N_j}(t) \Big) \Big( \prod_{j=n+1}^{m} N_j^{\frac{n}{m} -\frac{\theta}{2}} (c_{N_j}(t))^\theta (c'_{N_j}(t))^{1-\theta}\Big).
\]
For an arbitrary $\delta >0$, we set
\[
\tilde{c}_N(t) = \sum_{N' \in 2^{\Z}} \min(N/N', N'/N)^\delta c_{N'}(t),\quad \tilde{c}'_N(t) = \sum_{N' \in 2^{\Z}} \min(N/N', N'/N)^\delta c'_{N'}(t).
\]
Using the fact that $c_N(t) \leq \tilde{c}_N(t)$ and $\tilde{c}_{N_j}(t) \lesssim (N_1/N_j)^\delta\tilde{c}_{N_1}(t)$ for $j=2,...,m$ and similar estimates for primes, we see that
\begin{multline*}
A(t)\lesssim \sum_{N_1 \geq \cdots \geq N_m} \Big( \prod_{j=1}^{n}N_j^{\frac{n}{m} -\frac{1}{2}} (N_1/N_j)^\delta \tilde{c}_{N_1}(t) \Big) \\
\times \Big( \prod_{j=n+1}^{m} N_j^{\frac{n}{m}-\frac{\theta}{2}} (N_1/N_j)^\delta(\tilde{c}_{N_1}(t))^\theta (\tilde{c}'_{N_1}(t))^{1-\theta}\Big).
\end{multline*}
We can rewrite the above quantity in the right hand side as
\begin{multline*}
\sum_{N_1 \geq \cdots \geq N_m} \Big(\prod_{j=n+1}^{m}N_j^{\frac{n}{m}-\frac{\theta}{2}-\delta}\Big) \Big(\prod_{j=2}^{n} N_j^{\frac{n}{m}-\frac{1}{2}-\delta}\Big) N_1^{\frac{n}{m}-\frac{1}{2}+(m-1)\delta} (\tilde{c}_{N_1}(t))^{n+(m-n)\theta} \\
\times (\tilde{c}'_{N_1}(t))^{(m-n)(1-\theta)}.
\end{multline*}
By choosing $\theta = 1/(\nu-2) \in (0,1)$ and $\delta >0$ so that 
\[
\frac{n}{m}-\frac{\theta}{2}-\delta>0, \quad \frac{n}{m}-\frac{1}{2}+(m-1)\delta<0 \quad \text{or} \quad \delta <\frac{m-2n}{2m(m-1)}.
\]
Here condition $\nu>3$ ensures that $m-2n>0$. Summing in $N_m$, then in $N_{m-1}$,..., then in $N_2$, we have
\[
A(t) \lesssim \sum_{N_1 \in 2^{\Z}} (\tilde{c}_{N_1}(t))^{2n} (\tilde{c}'_{N_1}(t))^{(\nu-3)n}.
\]
The H\"older inequality with the fact that $(\nu-3)n \geq 1$ implies 
\begin{align}
A(t) &\lesssim \|(\tilde{c}(t))^{2n}\|_{\ell^2(2^{\Z})} \|(\tilde{c}'(t))^{(\nu-3)n}\|_{\ell^2(2^{\Z})} \nonumber \\
&= \|\tilde{c}(t)\|^{2n}_{\ell^{4n}(2^{\Z})} \|\tilde{c}'(t)\|^{(\nu-3)n}_{\ell^{2(\nu-3)n}(2^{\Z})} \leq \|\tilde{c}(t)\|^{2n}_{\ell^2(2^{\Z})}\|\tilde{c}'(t)\|^{(\nu-3)n}_{\ell^2(2^{\Z})}, \nonumber
\end{align}
where $\|\tilde{c}(t)\|_{\ell^q(2^\Z)}:=\Big(\sum_{N\in 2^\Z} |\tilde{c}_N(t)|^q\Big)^{1/q}$ and similarly for $\|\tilde{c}'(t)\|_{\ell^q(2^\Z)}$. The Minkowski inequality then implies
\[
A(t) \lesssim \|c(t)\|^{2n}_{\ell^2(2^{\Z})} \|c'(t)\|^{(\nu-3)n}_{\ell^2(2^{\Z})}.
\]
This implies that $A(t)<\infty$ for amost allwhere $t$, hence that $\sum_{N} \|P_Nu(t)\|_{L^\infty(\R^d)} <\infty$. Therefore $\sum_{N} P_Nu(t)$ converges in $L^\infty(\R^d)$. Since it converges to $u$ in the ditribution sense, so the limit is $u(t)$.
Thus
\begin{align}
\|u\|^{\nu-1}_{L^{\nu-1}(\R,L^\infty(\R^d))} &= \int_\R \|u(t)\|^{m/n}_{L^\infty(\R^d)} dt \lesssim \int_\R \|c(t)\|^2_{\ell^2(2^{\Z})} \|c'(t)\|^{\nu-3}_{\ell^2(2^{\Z})}dt \nonumber \\
&\lesssim \|c\|^2_{L^{2}_\R\ell^2(2^{\Z})} \|c'\|^{\nu-3}_{L^\infty_\R \ell^2(2^{\Z})} = \|u\|^2_{L^{2}(\R,\dot{B}^{\Gact-\gamma_{2,2^\star}}_{2^\star}(\R^d))}\|u\|^{\nu-3}_{L^\infty(\R,\dot{B}^{\Gact}_2(\R^d))}. \nonumber
\end{align}
The proof is complete.
\end{proof}
\subsection{Proof of Theorem $\ref{theorem subcritical}$}
We now give the proof of Theorem $\ref{theorem subcritical}$ by using the standard fixed point argument in a suitable Banach space. Thanks to $(\ref{subcritical assuption})$, we are able to choose $p> \max(\nu-1, 4)$ when $d=2$ and $p> \max(\nu-1, 2)$ when $d\geq 3$ such that $\gamma > d/2 - 1/p$ and then choose $q\in [2,\infty)$ such that
\[
\frac{2}{p} + \frac{d-1}{q} \leq  \frac{d-1}{2}.
\]
\textbf{Step 1.} Existence. Let us consider
\[
X := \Big\{ u \in L^\infty(I,H^\gamma) \cap L^p(I,H^{\gamma - \gamma_{p,q}}_q) \ | \ \|u\|_{L^\infty(I, H^\gamma)} + \|u\|_{L^p(I, H^{\gamma-\gamma_{p,q}}_q)} \leq M \Big\},
\]
equipped with the distance
\[
d(u,v):= \|u-v\|_{L^\infty(I, L^2)} + \|u-v\|_{L^p(I,H^{-\gamma_{p,q}}_q)},
\]
where $I =[0,T]$ and $M, T>0$ to be chosen later. By the Duhamel formula, it suffices to prove that the functional 
\begin{align}
\Phi(u)(t) = e^{it\Lambda} u_0 +i\mu \int_0^t e^{i(t-s)\Lambda} |u(s)|^{\nu-1} u(s)ds \label{duhamel formula schrodinger}
\end{align}
is a contraction on $(X,d)$. The Strichartz estimate $(\ref{strichartz sobolev})$ yields
\begin{align*}
\|\Phi(u)\|_{L^\infty(I,H^\gamma)}+\|\Phi(u)\|_{L^p(I, H^{\gamma-\gamma_{p,q}}_q)} &\lesssim \|u_0\|_{H^\gamma}+ \|F(u)\|_{L^1(I,H^\gamma)}, \\
\|\Phi(u)-\Phi(v)\|_{L^\infty(I,L^2)} +\|\Phi(u)-\Phi(v)\|_{L^p(I,H^{-\gamma_{p,q}}_q)} &\lesssim \|F(u)-F(v)\|_{L^1(I,L^2)},
\end{align*}
where $F(u) = |u|^{\nu-1} u$ and similarly for $F(v)$. By our assumptions on $\nu$, Corollary $\ref{coro fractional derivatives}$ gives
\begin{align}
\|F(u)\|_{L^1(I,H^\gamma)} &\lesssim \|u\|^{\nu-1}_{L^{\nu-1}(I,L^\infty)} \|u\|_{L^\infty(I,H^\gamma)} \nonumber \\
&\lesssim  T^{1-\frac{\nu-1}{p}} \|u\|^{\nu-1}_{L^p(I,L^\infty)} \|u\|_{L^\infty(I,H^\gamma)}, \label{subcritical 1} 
\end{align}
and
\begin{align}
\|F(u)-F(v)\|&_{L^1(I,L^2)} \lesssim \Big(\|u\|^{\nu-1}_{L^{\nu-1}(I,L^\infty)} + \|v\|^{\nu-1}_{L^{\nu-1}(I,L^\infty)}\Big) \|u-v\|_{L^\infty(I,L^2)} \nonumber \\
&\lesssim T^{1-\frac{\nu-1}{p}} \Big(\|u\|^{\nu-1}_{L^p(I,L^\infty)} + \|v\|^{\nu-1}_{L^p(I,L^\infty)}\Big) \|u-v\|_{L^\infty(I,L^2)}. \label{subcritical 2} 
\end{align}
The Sobolev embedding with the fact that $\gamma - \gamma_{p,q} > d/q$ implies $L^p(I,H^{\gamma - \gamma_{p,q}}_q) \subset L^p(I,L^\infty)$. Thus, we get
\[
\|\Phi(u)\|_{L^\infty(I,H^\gamma)}+\|\Phi(u)\|_{L^p(I, H^{\gamma-\gamma_{p,q}}_q)} \lesssim \|u_0\|_{H^\gamma}+ T^{1-\frac{\nu-1}{p}} \|u\|^{\nu-1}_{L^p(I,H^{\gamma-\gamma_{p,q}}_q)} \|u\|_{L^\infty(I,H^\gamma)}, 
\]
and
\[
d(\Phi(u),\Phi(v))  \lesssim T^{1-\frac{\nu-1}{p}} \Big(\|u\|^{\nu-1}_{L^p(I,H^{\gamma-\gamma_{p,q}}_q)} + \|v\|^{\nu-1}_{L^p(I,H^{\gamma-\gamma_{p,q}}_q)}\Big) \|u-v\|_{L^\infty(I,L^2)}. 
\]
This shows that for all $u,v \in X$, there exists $C>0$ independent of $u_0 \in H^\gamma$ and $T$ such that
\begin{align*}
\|\Phi(u)\|_{L^\infty(I,H^\gamma)}+\|\Phi(u)\|_{L^p(I, H^{\gamma-\gamma_{p,q}}_q)} &\leq C\|u_0\|_{H^\gamma} + C T^{1-\frac{\nu-1}{p}} M^{\nu}, \\
d(\Phi(u),\Phi(v))  &\leq C T^{1-\frac{\nu-1}{p}} M^{\nu-1} d(u,v).
\end{align*}
Therefore, if we set $M=2C\|u_0\|_{H^\gamma}$ and choose $T>0$ small enough so that $C T^{1-\frac{\nu-1}{p}} M^{\nu-1} \leq \frac{1}{2}$, then $X$ is stable by $\Phi$ and $\Phi$ is a contraction on $X$. By the fixed point theorem, there exists a unique $u \in X$ so that $\Phi(u)=u$. \newline
\noindent \textbf{Step 2.} Uniqueness. Consider $u, v \in C(I,H^\gamma) \cap L^p(I, L^\infty)$ two solutions of the (NLHW). Since the uniqueness is a local property (see \cite[Chapter 4]{Cazenave}), it suffices to show $u=v$ for $T$ is small. We have from $(\ref{subcritical 2})$ that
\[
d(u,v) \leq C T^{1-\frac{\nu-1}{p}} \Big( \|u\|^{\nu-1}_{L^p(I,L^\infty)} + \|v\|^{\nu-1}_{L^p(I,L^\infty)}\Big) d(u,v).
\] 
Since $\|u\|_{L^p(I, L^\infty)}$ is small if $T$ is small and similarly for $v$, we see that if $T>0$ small enough, 
\[
d(u,v) \leq \frac{1}{2} d(u,v) \text{ or } u=v.
\]
\noindent \textbf{Step 3.} Item i. Since the time of existence constructed in Step 1 only depends on $H^\gamma$-norm of the initial data. The blowup alternative follows by standard argument (see e.g. \cite[Chapter 4]{Cazenave}). \newline
\noindent \textbf{Step 4.} Item ii. Let $u_{0,n} \rightarrow u_0$ in $H^\gamma$ and $C, T=T(u_0)$ be as in Step 1. Set $M=4C\|u_0\|_{H^\gamma}$. It follows that $2C\|u_{0,n}\|_{H^\gamma} \leq M$ for sufficiently large $n$. Thus the solution $u_n$ constructed in Step 1 belongs to $X$ with $T=T(u_0)$ for $n$ large enough. We have from Strichartz estimate $(\ref{strichartz sobolev})$ and $(\ref{subcritical 1})$ that
\[
\|u\|_{L^a(I,H^{\gamma-\gamma_{a,b}}_b)} \lesssim \|u_0\|_{H^\gamma}+T^{1-\frac{\nu-1}{p}} \|u\|^{\nu-1}_{L^p(I,L^\infty)} \|u\|_{L^\infty(I,H^\gamma)},
\]
provided $(a,b)$ is admissible and $b<\infty$. This shows the boundedness of $u_n$ in $L^a(I,H^{\gamma-\gamma_{a,b}}_b)$. We also have from $(\ref{subcritical 2})$ and the choice of $T$ that
\[
d(u_n, u) \leq C\|u_{0,n}-u_0\|_{L^2}+\frac{1}{2}d(u_n,u) \text{ or } d(u_n, u) \leq 2C\|u_{0,n}-u_0\|_{L^2}.
\]
This yields that $u_n \rightarrow u$ in $L^\infty(I,L^2)\cap L^p(I,H^{-\gamma_{p,q}}_q)$. Strichartz estimate $(\ref{strichartz sobolev})$ again implies that $u_n\rightarrow u$ in $L^a(I,H^{-\gamma_{a,b}}_b)$ for any admissible pair $(a,b)$ with $b<\infty$. The convergence in $C(I,H^{\gamma-\ep})$ follows from the boundedness in $L^\infty(I,H^\gamma)$, the convergence in $L^\infty(I,L^2)$ and that $\|u\|_{H^{\gamma-\ep}} \leq \|u\|^{1-\frac{\ep}{\gamma}}_{H^\gamma}\|u\|^{\frac{\ep}{\gamma}}_{L^2}$. \newline
\noindent \textbf{Step 5.} Item iii. If $u_0 \in H^\beta$ for some $\beta >\gamma$ satisfying $\lceil \beta \rceil \leq \nu$ if $\nu>1$ is not an odd integer, then Step 1 shows the existence of $H^\beta$ solution defined on some maximal interval $[0,T)$. Since $H^\beta$ solution is also a $H^\gamma$ solution, thus $T\leq T^*$. Suppose that $T<T^*$. Then the unitary property of $e^{it\Lambda}$ and Lemma $ $ imply that
\[
\|u(t)\|_{H^\beta} \leq \|u_0\|_{H^\beta} + C \int_0^t \|u(s)\|^{\nu-1}_{L^\infty} \|u(s)\|_{H^\beta} ds,
\]
for all $0\leq t<T$. The Gronwall's inequality then gives
\[
\|u(t)\|_{H^\beta} \leq \|u_0\|_{H^\beta} \exp \Big( C \int_0^t \|u(s)\|^{\nu-1}_{L^\infty} ds \Big),
\]
for all $0\leq t<T$. Using the fact that  $u \in L^{\nu-1}_{\text{loc}}([0,T^*),L^\infty)$, we see that $\limsup \|u(t)\|_{H^\beta} <\infty$ as $t\rightarrow T$ which is a contradiction to the blowup alternative in $H^\beta$.
\defendproof
\begin{remark} \label{rem continuous dependence}
If we assume that $\nu>1$ is an odd integer or 
\[
\lceil \gamma \rceil \leq \nu-1
\]
otherwise, then the continuous dependence holds in $C(I,H^\gamma)$. To see this, we consider $X$ as above equipped with the following metric
\[
d(u,v):= \|u-v\|_{L^\infty(I,H^\gamma)} + \|u-v\|_{L^p(I,H^{\gamma-\gamma_{p,q}}_q)}.
\]
Using Item (ii) of Corollary $\ref{coro fractional derivatives}$, we have
\begin{multline*}
\|F(u)-F(v)\|_{L^1(I,H^\gamma)} \lesssim  (\|u\|^{\nu-1}_{L^{\nu-1}(I,L^\infty)} +\|v\|^{\nu-1}_{L^{\nu-1}(I,L^\infty)}) \|u-v\|_{L^\infty(I,H^\gamma)} \\ + (\|u\|^{\nu-2}_{L^{\nu-1}(I,L^\infty)}+\|v\|^{\nu-2}_{L^{\nu-1}(I,L^\infty)})(\|u\|_{L^\infty(I,H^\gamma)} + \|v\|_{L^\infty(I,H^\gamma)})\|u-v\|_{L^{\nu-1}(I,L^\infty)}.
\end{multline*}
The Sobolev embedding then implies for all $u, v \in X$,
\[
d(\Phi(u),\Phi(v)) \lesssim T^{1-\frac{\nu-1}{p}} M^{\nu-1} d(u,v).
\]
Therefore, the continuity in $C(I,H^\gamma)$ follows as in Step 4.
\end{remark}
\subsection{Proof of Theorem $\ref{theorem critical}$}
We now turn to the proof of the local well-posedness and small data scattering in critical case by following the same argument as in \cite{Dinh}.  \newline
\textbf{Step 1.} Existence. We only treat for $d\geq 4$, the ones for $d=2, d=3$ are completely similar. Let us consider
\begin{multline*}
X := \Big\{ u \in L^\infty(I,H^{\Gact}) \cap L^2(I,B^{\Gact-\gamma_{2,2^\star}}_{2^\star}) \ | \ 
\|u\|_{L^\infty(I,\dot{H}^{\Gact})} \leq M, \\
\|u\|_{L^2(I,\dot{B}^{\Gact-\gamma_{2,2^\star}}_{2^\star})} \leq N \Big\},
\end{multline*}
equipped with the distance
\[
d(u,v):= \|u-v\|_{L^\infty(I,L^2)} + \|u-v\|_{L^2(I,\dot{B}^{-\gamma_{2,2^\star}}_{2^\star})},
\]
where $I=[0,T]$ and $T, M, N>0$ will be chosen later. One can check (see e.g. \cite{CazenaveWeissler} or \cite{Cazenave}) that $(X,d)$ is a complete metric space. Using the Duhamel formula
\begin{align}
\Phi(u)(t) = e^{it\Lambda} u_0 +i\mu \int_0^t e^{i(t-s)\Lambda} |u(s)|^{\nu-1} u(s)ds =: u_{\text{hom}}(t)+ u_{\text{inh}}(t), 
\label{duhamel formula}
\end{align}
the Strichartz estimate $(\ref{homogeneous strichartz})$ yields
\[
\|u_{\text{hom}}\|_{L^2(I,\dot{B}^{\Gact-\gamma_{2,2^\star}}_{2^\star})} \lesssim \|u_0\|_{\dot{H}^{\Gact}}.
\]
A similar estimate holds for $\|u_{\text{hom}}\|_{L^\infty(I,\dot{H}^{\Gact})}$. We see that $\|u_{\text{hom}}\|_{L^2(I,\dot{B}^{\Gact-\gamma_{2,2^\star}}_{2^\star})} \leq \varepsilon$ for some $\varepsilon>0$ small enough which will be chosen later, provided that either $\|u_0\|_{\dot{H}^{\Gact}}$ is small or it is satisfied some $T>0$ small enough by the dominated convergence theorem. Therefore, we can take $T=\infty$ in the first case and $T$ be this finite time in the second. On the other hand, using again $(\ref{homogeneous strichartz})$, we have
\[
\|u_{\text{inh}}\|_{L^2(I,\dot{B}^{\Gact-\gamma_{2,2^\star}}_{2^\star})} \lesssim \|F(u)\|_{L^1(I,\dot{H}^{\Gact})}.
\]
A same estimate holds for $\|u_{\text{inh}}\|_{L^\infty(I,\dot{H}^{\Gact})}$. Corollary $\ref{coro fractional derivatives}$ and Lemma $\ref{lem nonlinear control}$ give
\begin{align}
\|F(u)\|_{L^1(I,\dot{H}^{\Gact})} &\lesssim  \|u\|^{\nu-1}_{L^{\nu-1}(I,L^\infty)} \|u\|_{L^\infty(I,\dot{H}^{\Gact})} \nonumber \\
&\lesssim \|u\|^2_{L^2(I,\dot{B}^{\Gact-\gamma_{2,2^\star}}_{2^\star})}\|u\|^{\nu-2}_{L^\infty(I,\dot{H}^{\Gact})}. \label{critical 1} 
\end{align}
Similarly, we have
\begin{align}
\|F(u)-F(v)&\|_{L^1(I,L^2)} \lesssim \Big(\|u\|^{\nu-1}_{L^{\nu-1}(I,L^\infty)} + \|v\|^{\nu-1}_{L^{\nu-1}(I,L^\infty)} \Big) \|u-v\|_{L^\infty(I,L^2)} \label{critical 2} \\
&\lesssim   \Big(\|u\|^2_{L^2(I,\dot{B}^{\Gact-\gamma_{2,2^\star}}_{2^\star})} \|u\|^{\nu-3}_{L^\infty(I,\dot{H}^{\Gact})} + \|v\|^2_{L^2(I,\dot{B}^{\Gact-\gamma_{2,2^\star}}_{2^\star})} \|v\|^{\nu-3}_{L^\infty(I,\dot{H}^{\Gact})}\Big)  \nonumber \\
&\mathrel{\phantom{\|_{L^1(I,L^2)} \lesssim \Big(\|u\|^{\nu-1}_{L^{\nu-1}(I,L^\infty)} + \|v\|^{\nu-1}_{L^{\nu-1}(I,L^\infty)} \Big)}} \times \|u-v\|_{L^\infty(I,L^2)}.\nonumber
\end{align}
This implies for all $u, v \in X$, there exists $C>0$ independent of $u_0 \in H^{\Gact}$ such that
\begin{align*}
\|\Phi(u)\|_{L^2(I,\dot{B}^{\Gact-\gamma_{2,2^\star}}_{2^\star})} &\leq \varepsilon +CN^2M^{\nu-2}, \\
\|\Phi(u)\|_{L^\infty(I,\dot{H}^{\Gact})} &\leq C\|u_0\|_{\dot{H}^{\Gact}} +CN^2M^{\nu-2}, \\
d(\Phi(u),\Phi(v)) &\leq CN^2M^{\nu-3} d(u,v).
\end{align*}
Now by setting $N=2\varepsilon$ and $M=2C\|u_0\|_{\dot{H}^{\Gact}}$ and choosing $\varepsilon>0$ small enough such that $CN^2M^{\nu-3} \leq \min\{1/2, \varepsilon/M \}$, we see that $X$ is stable by $\Phi$ and $\Phi$ is a contraction on $X$. By the fixed point theorem, there exists a unique solution $u \in X$ to the (NLHW). Note that when $\|u_0\|_{\dot{H}^{\Gact}}$ is small enough, we can take $T=\infty$. \newline
\textbf{Step 2.} Uniqueness. The uniqueness in $C^\infty(I,H^{\Gact}) \cap L^2(I, B^{\Gact-\gamma_{2,2^\star}}_{2^\star})$ follows as in Step 2 of the proof of Theorem $\ref{theorem subcritical}$ using $(\ref{critical 2})$. Here $\|u\|_{L^2(I,\dot{B}^{\Gact-\gamma_{2,2^\star}}_{2^\star})}$ can be small as $T$ is small.\newline
\textbf{Step 3.} Scattering. The global existence when $\|u_0\|_{\dot{H}^{\Gact}}$ is small is given in Step 1. It remains to show the scattering property. Thanks to $(\ref{critical 1})$, we see that
\begin{align}
\|e^{-it_2\Lambda}u(t_2)- e^{-it_1\Lambda}u(t_1)\|_{\dot{H}^{\Gact}} &=\Big\|i\mu \int_{t_1}^{t_2} e^{-is\Lambda} (|u|^{\nu-1}u)(s) ds\Big\|_{\dot{H}^{\Gact}} \nonumber \\
&\leq \|F(u)\|_{L^1([t_1,t_2], \dot{H}^{\Gact})} \nonumber \\
&\lesssim \|u\|^2_{L^2([t_1,t_2],\dot{B}^{\Gact-\gamma_{2,2^\star}}_{2^\star})}\|u\|^{\nu-2}_{L^\infty([t_1,t_2],\dot{H}^{\Gact})} \rightarrow 0 \label{scattering 1}
\end{align}
as $t_1, t_2 \rightarrow + \infty$. We have from $(\ref{critical 2})$ that
\begin{align}
\|e^{-it_2\Lambda}u(t_2)- e^{-it_1\Lambda}u(t_1)\|_{L^2} &\lesssim  \|u\|^2_{L^2([t_1,t_2],\dot{B}^{\Gact-\gamma_{2,2^\star}}_{2^\star})}\|u\|^{\nu-3}_{L^\infty([t_1,t_2],\dot{H}^{\Gact})} \nonumber \\
& \mathrel{\phantom{\lesssim  \|u\|^2_{L^2([t_1,t_2],\dot{B}^{\Gact-\gamma_{2,2^\star}}_{2^\star})}}} \times \|u\|_{L^\infty([t_1,t_2],L^2)}, \label{scattering 2}
\end{align}
which also tends to zero as $t_1, t_2 \rightarrow +\infty$. This implies that the limit 
\[
u_0^+:= \lim_{t \rightarrow + \infty} e^{-it\Lambda} u(t)
\]
exists in $H^{\Gact}$. Moreover, we have
\[
u(t)-e^{it\Lambda}u_0^+ =-i\mu\int_{t}^{+ \infty} e^{i(t-s)\Lambda}F(u(s)) ds. 
\]
The unitary property of $e^{it\Lambda}$ in $L^2$, $(\ref{scattering 1})$ and $(\ref{scattering 2})$ imply that $\|u(t)-e^{it\Lambda}u_0^+\|_{H^{\Gact}} \rightarrow 0$ when $t \rightarrow + \infty$. This completes the proof of Theorem $\ref{theorem critical}$.
\defendproof
\section{Ill-posedness}
In this section, we will give the proof of Theorem $\ref{theorem ill-posedness}$. We follow closely the argument of \cite{ChristCollianderTao} using small dispersion analysis and decoherence arguments.
\subsection{Small dispersion analysis}
Now let us consider for $0<\delta \ll 1$ the following equation
\begin{align}
\left\{
\begin{array}{rcl}
i\partial_t \phi(t,x) + \delta\Lambda \phi(t,x)&=&-\mu |\phi|^{\nu-1} \phi(t,x), \quad (t, x) \in \R \times \R^d, \\
\phi(0,x) &=& \phi_0(x), \quad x\in \R^d.
\end{array}
\right.
\label{small dispersion}
\end{align}
Note that $(\ref{small dispersion})$ can be transformed back to the (NLHW) by using
\[
u(t,x):= \phi(t,\delta x).
\]
\begin{lemma} \label{lem small dispersion analysis}
Let $k>d/2$ be an integer. If $\nu$ is not an odd integer, then we assume also the additional regularity condition $\nu \geq k+1$. Let $\phi_0$ be a Schwartz function. Then there exists $C, c>0$ such that if $0<\delta\leq c$ sufficiently small, then there exists a unique solution $\phi^{(\delta)} \in C([-T,T],H^k)$ of $(\ref{small dispersion})$ with $T=c|\log \delta|^c$ satisfying
\begin{align}
\|\phi^{(\delta)}(t)-\phi^{(0)}(t)\|_{H^k} \leq C\delta^{1/2}, \label{small dispersion estimate}
\end{align}
for all $|t|\leq c|\log \delta|^c$, where 
\[
\phi^{(0)}(t,x):= \phi_0(x) \exp(-i\mu t|\phi_0(x)|^{\nu-1})
\]
is the solution of $(\ref{small dispersion})$ with $\delta=0$.
\end{lemma}
\begin{proof}
We refer the reader to \cite[Lemma 2.1]{ChristCollianderTao} where the small dispersion analysis is invented to prove the ill-posedness for the classical nonlinear Schr\"odinger equation. The same proof can be applied to the nonlinear half-wave equation without any difficulty. By using the energy method, we end up with the following estimate
\[
\|\phi^{(\delta)}(t)-\phi^{(0)}(t)\|_{H^k} \leq C \delta \exp(C(1+|t|)^C).
\]
Thus, if $|t| \leq c|\log \delta|^c$ for suitably small $0<\delta \leq c$, then $\exp(C(1+|t|)^C) \leq \delta^{-1/2}$ and $(\ref{small dispersion estimate})$ follows.
\end{proof}
\begin{remark} \label{rem small dispersion analysis}
By the same argument as in \cite{ChristCollianderTao}, we can get the following better estimate
\begin{align}
\|\phi^{(\delta)}(t)-\phi^{(0)}(t)\|_{H^{k,k}} \leq C \delta^{1/2}, \label{small dispersion estimate weighted}
\end{align}
for all $|t| \leq c|\log \delta|^c$, where $H^{k,k}$ is the weighted Sobolev space
\[
\|\phi\|_{H^{k,k}}:= \sum_{|\alpha|=0}^{k} \|\scal{x}^{k-|\alpha|} D^\alpha \phi\|_{L^2}. 
\]
\end{remark}
Now let $\lambda>0$ and set
\begin{align}
u^{(\delta, \lambda)}(t,x):= \lambda^{-\frac{1}{\nu-1}} \phi^{(\delta)}(\lambda^{-1}t, \lambda^{-1} \delta x). \label{define solution}
\end{align}
It is easy to see that $u^{(\delta,\lambda)}$ is a solution of the (NLHW). 
\begin{lemma} \label{lem initial data estimate}
Let $\gamma \in \R$ and $0<\lambda \leq \delta \ll 1$. Let $\phi_0 \in \Sh$ be such that if $\gamma \leq -d/2$,
\[
\hat{\phi}_0(\xi) = O(|\xi|^\kappa) \text{ as } \xi \rightarrow 0,
\]
for some $\kappa >-\gamma-d/2$, where $\hat{\cdot}$ is the Fourier transform. Then there exists $C>0$ such that 
\begin{align}
\|u^{(\delta,\lambda)}(0)\|_{H^\gamma} \leq C \lambda^{\Gace-\gamma} \delta^{\gamma-d/2}. \label{initial data estimate}
\end{align}
\end{lemma}
\begin{proof}
The proof of this lemma is essentially given in \cite{ChristCollianderTao}. For reader's convenience, we give a sketch of the proof. We firstly have
\[
[u^{(\delta,\lambda)}(0)]\hat{\ } (\xi)= \lambda^{-\frac{1}{\nu-1}} (\lambda\delta^{-1})^d \hat{\phi}_0(\lambda\delta^{-1}\xi).
\]
Thus,
\begin{align*}
\|u^{(\delta,\lambda)}(0)\|^2_{H^\gamma}&= \lambda^{-\frac{2}{\nu-1}} (\lambda\delta^{-1})^{2d} \int (1+|\xi|^2)^\gamma |\hat{\phi}_0(\lambda\delta^{-1}\xi)|^2 d\xi \\
&=\lambda^{-\frac{2}{\nu-1}} (\lambda\delta^{-1})^d \int (1+|\lambda^{-1}\delta\xi|^2)^\gamma |\hat{\phi}_0(\xi)|^2 d\xi \\
&\sim \lambda^{-\frac{2}{\nu-1}} (\lambda\delta^{-1})^{d-2\gamma} \int_{|\xi| \geq \lambda\delta^{-1} } |\xi|^{2\gamma} |\hat{\phi}_0(\xi)|^2 d\xi \\
&\mathrel{\phantom{\sim \lambda^{-\frac{2}{\nu-1}} (\lambda\delta^{-1})^{d-2\gamma} \int_{|\xi| \geq \lambda\delta^{-1}}}} + \lambda^{-\frac{8}{\nu-1}} (\lambda\delta^{-1})^d \int_{|\xi| \leq \lambda\delta^{-1} } |\hat{\phi}_0(\xi)|^2 d\xi \\
&= \lambda^{-\frac{2}{\nu-1}} (\lambda\delta^{-1})^{d-2\gamma} \Big(\int_{\R} |\xi|^{2\gamma}|\hat{\phi}_0(\xi)|^2d\xi \\
&\mathrel{\phantom{= \lambda^{-\frac{2}{\nu-1}} (\lambda\delta^{-1})^{d-2\gamma} \Big(\int_{\R} }} - \int_{|\xi|\leq \lambda\delta^{-1}} ( (\lambda\delta^{-1})^{2\gamma} - |\xi|^{2\gamma} )|\hat{\phi}_0(\xi)|^2 d\xi  \Big).
\end{align*}
Using the fact that $\lambda\delta^{-1} \leq 1$, we obtain for $\gamma>-d/2$ that
\[
\|u^{(\delta,\lambda)}(0)\|_{H^\gamma} = c\lambda^{-\frac{1}{\nu-1}} (\lambda\delta^{-1})^{d/2-\gamma} (1+O((\lambda\delta^{-1})^{\gamma+d/2})) \leq C \lambda^{\Gact-\gamma} \delta^{\gamma-d/2},
\]
where $c\ne 0$ provided that $\phi_0$ is not identically zero. Moreover, for $\gamma\leq -d/2$, the assumption on $\hat{\phi}_0$ also implies
\[
\|u^{(\delta,\lambda)}(0)\|_{H^\gamma} \leq C \lambda^{\Gact-\gamma} \delta^{\gamma-d/2}.
\]
Here we use the fact that
\[
\int_{|\xi|\leq \lambda\delta^{-1}} ( (\lambda\delta^{-1})^{2\gamma} - |\xi|^{2\gamma} )|\hat{\phi}_0(\xi)|^2 d\xi \leq C (\lambda \delta^{-1})^{d+2\gamma+2\kappa} \leq C.
\]
This completes the proof of $(\ref{initial data estimate})$.
\end{proof}
\subsection{Proof of Theorem $\ref{theorem ill-posedness}$}
We are now able to prove Theorem $\ref{theorem ill-posedness}$. We only consider the case $t\geq 0$, the one for $t<0$ is similar. Let $\ep \in (0,1]$ be fixed and set
\begin{align}
\lambda^{\Gact-\gamma} \delta^{\gamma-d/2} =:\ep, \label{define epsilon}
\end{align}
equivalently
\[
\lambda=\delta^\theta, \text{ where } \theta=\frac{d/2-\gamma}{\Gact-\gamma}>1.
\]
Note that we are considering here $\gamma<\Gact$. This implies that $0 < \lambda \leq \delta \ll 1$, and Lemma $\ref{lem initial data estimate}$ gives
\[
\|u^{(\delta,\lambda)}(0)\|_{H^\gamma} \leq C\ep.
\]
We now split the proof to several cases.
\paragraph{\bf The case $0<\gamma<\Gact$.} Since the support of $\phi^{(0)}(t,x)$ is independent of $t$, we see that for $t$ large enough, depending on $\gamma$,
\[
\|\phi^{(0)}(t)\|_{H^\gamma} \sim t^\gamma,
\]
whenever $\gamma \geq 0$ provided either $\nu>1$ is an odd integer or $\gamma \leq \nu-1$ otherwise. Thus for $\delta \ll 1$ and $1\ll t \leq c|\log \delta|^c$, $(\ref{small dispersion estimate})$ implies
\begin{align}
\|\phi^{(\delta)}(t)\|_{H^\gamma} \sim t^\gamma. \label{solution estimate}
\end{align}
We next have 
\[
[u^{(\delta,\lambda)}(\lambda t)]\hat{\ }(\xi)= \lambda^{-\frac{1}{\nu-1}} (\lambda\delta^{-1})^d [\phi^{(\delta)}(t)]\hat{\ } (\lambda\delta^{-1}\xi).
\]
This shows that
\begin{align*}
\|u^{(\delta,\lambda)}(\lambda t)\|^2_{H^\gamma} &= \int (1+|\xi|^2)^\gamma |[u^{(\delta,\lambda)}(\lambda t)]\hat{\ } (\xi)|^2 d\xi \\
&= \lambda^{-\frac{2}{\nu-1}} (\lambda\delta^{-1})^d \int (1+|\lambda^{-1}\delta\xi|^2)^\gamma |[\phi^{(\delta)}(t)]\hat{\ } (\xi)|^2 d\xi \\
&\geq \lambda^{-\frac{2}{\nu-1}} (\lambda\delta^{-1})^{d-2\gamma} \int_{|\xi|\geq 1} |\xi|^{2\gamma} |[\phi^{(\delta)}(t)]\hat{\ }(\xi)|^2 d\xi \\
&\geq \lambda^{-\frac{2}{\nu-1}} (\lambda\delta^{-1})^{d-2\gamma} \Big(c\|\phi^{(\delta)}(t)\|^2_{H^\gamma} - C\|\phi^{(\delta)}(t)\|^2_{L^2} \Big).
\end{align*}
Thanks to $(\ref{solution estimate})$, we have $\|\phi^{(\delta)}(t)\|_{L^2} \ll \|\phi^{(\delta)}(t)\|_{H^\gamma}$ for $t \gg 1$. This yields that
\[
\|u^{(\delta,\lambda)}(\lambda t)\|_{H^\gamma} \geq c \lambda^{-\frac{1}{\nu-1}} (\lambda\delta^{-1})^{d/2-\gamma} \|\phi^{(\delta)}(t)\|_{H^\gamma} \geq c \ep t^\gamma,
\]
for $1\ll t \leq c|\log \delta|^c$. We now choose $t=c|\log \delta|^c$ and pick $\delta>0$ small enough so that
\[
\ep t^\gamma > \ep^{-1}, \quad \lambda t < \ep.
\]
Therefore, for any $\varepsilon>0$, there exists a solution of the (NLHW) satisfying 
\[
\|u(0)\|_{H^\gamma} <\varepsilon, \quad \|u(t)\|_{H^\gamma}>\varepsilon^{-1}
\] 
for some $t \in (0,\varepsilon)$. Thus for any $t>0$, the solution map $\Sh \ni u(0) \mapsto u(t)$ for the Cauchy problem (NLHW) fails to be continuous at 0 in the $H^\gamma$-topology.
\paragraph{\bf The case $\gamma=0 <\Gact$.} Let $a, a' \in [1/2,2]$. Let $\phi^{(a,\delta)}$ be the solution to $(\ref{small dispersion})$ with initial data
\[
\phi^{(a,\delta)}(0)=a \phi_0.
\]
Then, Lemma $\ref{lem small dispersion analysis}$ gives
\begin{align}
\|\phi^{(a,\delta)}(t)-\phi^{(a,0)}(t)\|_{H^k} \leq C \delta^{1/2}, \label{small dispersion estimate phi a delta}
\end{align}
for all $|t|\leq c|\log \delta|^c$, where 
\begin{align}
\phi^{(a,0)}(t,x)= a\phi_0(x) \exp (-i\mu a^{\nu-1} t|\phi_0(x)|^{\nu-1}) \label{define phi a zero}
\end{align} 
is the solution of $(\ref{small dispersion})$ with $\delta=0$ and the same initial data as $\phi^{(a,\delta)}$. Note that since $a$ belongs to a compact set, the constant $C, c$ can be taken to be independent of $a$. We next define
\begin{align}
u^{(a,\delta,\lambda)}(t,x):= \lambda^{-\frac{1}{\nu-1}} \phi^{(a,\delta)}(\lambda^{-1}t, \lambda^{-1}\delta x). \label{define u a delta lambda}
\end{align}
It is easy to see that $u^{(a,\delta, \lambda)}$ is also a solution of the (NLHW). 
Using $(\ref{define phi a zero})$, a direct computation shows that
\[
\|\phi^{(a,0)}(t)-\phi^{(a',0)}(t)\|_{L^2} \geq c >0,
\]
for some time $t$ satisfying $|a-a'|^{-1}\leq t\leq c|\log \delta|^c$ provided that $\delta$ is small enough so that $c|\log \delta|^c \geq |a-a'|^{-1}$. The triangle inequality together with $(\ref{small dispersion estimate phi a delta})$ yield
\[
\|\phi^{(a,\delta)}(t)-\phi^{(a',\delta)}(t)\|_{L^2} \geq c,
\]
for all $|a-a'|^{-1} \leq t \leq c|\log \delta|^c$. Now let $\ep$ be as in $(\ref{define epsilon})$, i.e.
\[
\lambda^{-\frac{1}{\nu-1}}(\lambda\delta^{-1})^{d/2} =:\ep,
\]
or $\lambda=\delta^\theta$ with $\theta=\frac{d/2}{\Gact}>1$. Moreover, using the fact 
\[
[u^{(a,\delta,\lambda)}(\lambda t)]\hat{\ } (\xi) = \lambda^{-\frac{1}{\nu-1}} (\lambda\delta^{-1})^{d}[\phi^{(a,\delta)}(t)]\hat{\ }(\lambda\delta^{-1}\xi),
\]
we have
\[
\|u^{(a,\delta,\lambda)}(\lambda t)-u^{(a',\delta,\lambda)}(\lambda t)\|_{L^2} = \lambda^{-\frac{1}{\nu-1}}(\lambda\delta^{-1})^{d/2} \|\phi^{(a,\delta)}(t)-\phi^{(a',\delta)}(t)\|_{L^2} \geq c \ep. 
\]
Similarly, using that
\[
[u^{(a,\delta,\lambda)}(0)]\hat{\ } (\xi)= a \lambda^{-\frac{1}{\nu-1}} (\lambda\delta^{-1})^d \hat{\phi}_0(\lambda \delta^{-1}\xi),
\]
we have
\[
\|u^{(a,\delta,\lambda)}(0)\|_{L^2}, \|u^{(a',\delta,\lambda)}(0)\|_{L^2} \leq C \ep,
\]
and
\[
\|u^{(a,\delta,\lambda)}(0)-u^{(a',\delta,\lambda)}(0)\|_{L^2} \leq C\ep|a-a'|.
\]
Since $|a-a'|$ can be arbitrarily small, this shows that for any $0<\ep, \sigma <1$ and for any $t>0$, there exist $u_1, u_2$ solutions of the (NLHW) with initial data $u_1(0), u_2(0) \in \Sh$ such that
\[
\|u_1(0)\|_{L^2}, \|u_2(0)\|_{L^2} \leq C\ep, \quad \|u_1(0)-u_2(0)\|_{L^2} \leq C\sigma, \quad \|u_1(t)-u_2(t)\|_{L^2} \geq c\ep.
\]
This shows that the solution map fails to be uniformly continuous on $L^2$.
\paragraph{\bf The case $\gamma \leq -d/2$ and $\gamma<\Gact$.} 
Let $u^{(\delta,\lambda)}$ be as in $(\ref{define solution})$. Thanks to $(\ref{define epsilon})$, we have 
\[
\|u^{(\delta,\lambda)}(0)\|_{H^\gamma} \leq C\ep,
\]
provided $0<\lambda\leq \delta \ll 1$ and $\phi_0 \in \Sh$ satisfying
\[
\hat{\phi}_0(\xi) = O(|\xi|^\kappa) \text{ as } \xi \rightarrow 0,
\]
for some $\kappa >-\gamma-d/2$. We recall that
\[
\phi^{(0)}(t,x)=\phi_0(x) \exp(-i\mu t|\phi_0(x)|^{\nu-1}).
\]
It is clear that we can choose $\phi_0$ so that
\[
\Big|\int \phi^{(0)}(1,x) dx \Big| \geq c \text{ or } |[\phi^{(0)}(1)]\hat{\ }(0)| \geq c,
\]
for some constant $c>0$. Since $\phi^{(0)}(1)$ is rapidly decreasing, the continuity implies that
\[
|[\phi^{(0)}(1)]\hat{\ } (\xi)| \geq c,
\]
for $|\xi| \leq c$ with $0<c \ll 1$. On the other hand, using $(\ref{small dispersion estimate weighted})$ (note that $H^{k,k}$ controls $L^1$ when $k>d/2$), we have
\[
|[\phi^{(\delta)}(1)]\hat{\ } (\xi) - [\phi^{(0)}(1)]\hat{\ } (\xi)| \leq C \delta^{1/2},
\]
and then
\[
|[\phi^{(\delta)}(1)]\hat{\ }(\xi)| \geq c,
\]
for $|\xi|\leq c$ provided $\delta$ is taken small enough. Moreover, we have
\[
u^{(\delta,\lambda)}(\lambda,x)= \lambda^{-\frac{1}{\nu-1}} \phi^{(\delta)}(1,\lambda^{-1}\delta x)
\] 
and 
\[
[u^{(\delta,\lambda)}(\lambda)]\hat{\ }(\xi) = \lambda^{-\frac{1}{\nu-1}} (\lambda\delta^{-1})^d[\phi^{(\delta)}(1)]\hat{\ }(\lambda\delta^{-1}\xi).
\]
This implies that
\[
[u^{(\delta,\lambda)}(\lambda)]\hat{\ }(\xi) \geq c \lambda^{-\frac{1}{\nu-1}} (\lambda\delta^{-1})^d,
\]
for $|\xi|\leq c \lambda^{-1}\delta$. \newline
\indent \underline{In the case $\gamma<-d/2$}, we have
\[
\|u^{(\delta,\lambda)}(\lambda)\|_{H^\gamma} \geq c \lambda^{-\frac{1}{\nu-1}} (\lambda\delta^{-1})^d = c \ep (\lambda \delta^{-1})^{\gamma+d/2}.
\]
Here $0<\lambda\leq \delta \ll 1$, thus $(\lambda\delta^{-1})^{\gamma+d/2}\rightarrow +\infty$. We can choose $\delta$ small enough so that $\lambda \rightarrow 0$ and $(\lambda\delta^{-1})^{\gamma+d/2}\geq \ep^{-2}$ or
\[
\|u^{(\delta,\lambda)}(\lambda)\|_{H^\gamma} \geq \ep^{-1}.
\]
This shows that the solution map fails to be continuous at 0 in $H^\gamma$-topology. \newline
\indent \underline{In the case $\gamma=-d/2$}, we have
\begin{align*}
\|u^{(\delta,\lambda)}(\lambda)\|_{H^{-d/2}} &\geq c \lambda^{-\frac{1}{\nu-1}} (\lambda\delta^{-1})^d \Big(\int_{|\xi|\leq c \lambda^{-1} \delta} (1+|\xi|)^{-d} d\xi\Big)^{1/2} \\
&=c \lambda^{-\frac{1}{\nu-1}} (\lambda\delta^{-1})^d (\log (c\lambda^{-1}\delta))^{1/2} \\
&=c \ep (\log (c\lambda^{-1}\delta))^{1/2}.
\end{align*}
By choosing $\delta$ small enough so that $\lambda \rightarrow 0$ and $\log (c\lambda^{-1}\delta) \geq \ep^{-4}$, we see that
\[
\|u^{(\delta,\lambda)}(\lambda)\|_{H^{-d/2}} \geq \ep^{-1}.
\]
This completes the proof of Thereom $\ref{theorem ill-posedness}$.
\defendproof
\section*{Acknowledgments} The author would like to express his deep thanks to his wife-Uyen Cong for her encouragement and support. He also would like to thank his supervisor Prof. Jean-Marc BOUCLET for the kind guidance and constant encouragement. He also would like to thank the reviewers for their helpful comments and suggestions, which helped improve the manuscript.
\addcontentsline{toc}{section}{Acknowledments}

{\sc Institut de Math\'ematiques de Toulouse, Universit\'e Toulouse III Paul Sabatier, 31062 Toulouse Cedex 9, France.} \\
\indent Email: \href{mailto:dinhvan.duong@math.univ-toulouse.fr}{dinhvan.duong@math.univ-toulouse.fr}
\end{document}